\documentclass[12pt]{article}
\usepackage[top=1in,bottom=1in,left=1in,right=1in]{geometry}
\usepackage{indentfirst}
\usepackage{amsfonts,amsmath,amsthm,amssymb}
\usepackage{longtable,enumerate,xcolor,seqsplit}
\usepackage{authblk}
\usepackage[pagebackref,colorlinks=true,citecolor=blue,linkcolor=red,urlcolor=blue]{hyperref}

\newtheorem{theorem}{Theorem}[section] 
\newtheorem{definition}[theorem]{Definition}

\newtheorem{lemma}[theorem]{Lemma}

\newtheorem{corollary}[theorem]{Corollary} 
\newtheorem{remark}[theorem]{Remark}

\newtheorem{question}{Question}

\theoremstyle{plain}

\renewcommand{\leq}{\leqslant}
\renewcommand{\geq}{\geqslant}

\renewcommand{\le}{\leqslant}
\renewcommand{\ge}{\geqslant}

\newcommand{\PSL}{{\rm PSL}}
\newcommand{\PGL}{{\rm PGL}}

\newcommand{\PGU}{{\rm PGU}}
\newcommand{\PSU}{{\rm PSU}}
\newcommand{\PSp}{{\rm PSp}}
\newcommand{\PGO}{{\rm PGO}}
\newcommand{\POm}{{\rm P}\Omega}

\DeclareMathOperator{\Out}{Out}

\newcommand\Om{\Omega}

\newcommand\Alt{{\rm Alt}}

\newcommand\Sym{{\rm Sym}}

\newcommand\Omk{{(k),\Om}}

\begin{document}

\title{Total closure for permutation actions of\\finite nonabelian simple groups}

\author[1]{Saul D. Freedman}
\author[2]{Michael Giudici}
\author[2]{Cheryl E. Praeger}

\affil[1]{School of Mathematics and Statistics, University of St Andrews}
\affil[2]{Centre for the Mathematics of Symmetry and Computation, Department of Mathematics and Statistics, The University of Western Australia}

\maketitle
 
\begin{abstract}
For a positive integer $k$, a group $G$ is said to be totally $k$-closed if for each set $\Omega$ upon which $G$ acts faithfully, $G$ is the largest subgroup of $\Sym(\Om)$ that leaves invariant each of the $G$-orbits in the induced action on $\Om\times\cdots\times \Om=\Om^k$. Each finite group $G$ is totally $|G|$-closed, and $k(G)$ denotes the least integer $k$ such that $G$ is totally $k$-closed. We address the question of determining the closure number $k(G)$ for finite simple groups $G$. Prior to our work it was known that $k(G)=2$ for cyclic groups of prime order and for precisely six of the sporadic simple groups, and that $k(G)\geq3$ for all other finite simple groups. We determine the value for the alternating groups, namely $k(A_n)=n-1$. In addition, for all simple groups $G$, other than alternating groups and classical groups,  we show that $k(G)\leq 7$. Finally, if $G$ is a finite simple classical group with natural module of dimension $n$, we show that $k(G)\leq n+2$ if $n \ge 14$, and $k(G) \le \lfloor n/3 + 12 \rfloor$ otherwise, with smaller bounds achieved by certain families of groups. This is achieved by determining a uniform upper bound (depending on $n$ and the type of $G$) on the base sizes of the primitive actions of $G$, based on known bounds for specific actions. We pose several open problems aimed at completing the determination of the closure numbers for finite simple groups.
\end{abstract}

\section{Introduction}

The notion of $k$-closure of a permutation group $G$ on a set $\Omega$ was introduced  in 1969 by Wielandt~\cite[Definition 5.3]{W}: namely the \emph{$k$-closure} $G^\Omk$  of $G$ is the set of all  $g\in\Sym(\Om)$ (permutations of $\Om$) such that $g$ leaves invariant each $G$-orbit in the induced $G$-action on ordered $k$-tuples from $\Om$.  The $k$-closure $G^{(k),\Omega}$ is a subgroup of  $\Sym(\Om)$ containing $G$ \cite[Theorem 5.4]{W}, and a permutation group $G$ is said to be \emph{$k$-closed} if $G^{(k),\Omega}=G$. Determining when $G$ is $k$-closed can be important when studying the \emph{relational complexity} of $G$. This is the smallest positive integer $r$ such that, whenever $A$ and $B$ are ordered $n$-tuples from $\Om$ with $n \ge r$, and each $r$-subtuple of $A$ lies in the same $G$-orbit as the corresponding $r$-subtuple of $B$, it is also the case that $A \in B^G$. The relational complexity of $G$ is very difficult to calculate in general, but it is bounded below by the smallest integer $k$ such that $G$ is $k$-closed, as shown in \cite[Lemmas 1.6.5]{GLS}; see that monograph for an excellent introduction to relational complexity.

The $k$-closures for different faithful permutation representations of the same group $G$ may be quite different, even for the symmetric group $S_3$ (see~\cite[p.\,240]{CP}). 
We are interested in the stronger concept of \emph{total $k$-closure}, which is independent of the permutation representation of the group.
A group $G$ is said to be \emph{totally $k$-closed} if for each set $\Omega$ upon which $G$ acts faithfully we have  
$G=G^{(k),\Omega}$. The concept was suggested in 2016   by Holt~\cite{Holt}, and first studied  by Abdollahi and Arezoomand~\cite{AA} when $k=2$. 

For a permutation group $G\leq\Sym(\Omega)$, and $k\geq 2$, Wielandt~\cite[Theorem 5.8]{W} proved that
\begin{equation}\label{eq1}
G\leq G^{(k),\Omega}\leq G^{(k-1),\Omega},
\end{equation}
and in addition, see \cite[Theorem 5.12]{W}, that if $G$ is finite  then $G=G^{(k),\Omega}$ for 
$k=|G|$. Thus if a finite group $G$ is totally $(k-1)$-closed then it is automatically totally $k$-closed, and also $G$ is totally $k$-closed for sufficiently large $k$. We define the \emph{closure number} $k(G)$ of a finite group $G$ as the least positive integer $k$ such that $G$ is totally $k$-closed. 

In this paper we study the closure numbers for finite simple groups. In previous work it has been shown that $k(G)=2$ if $G$ is a cyclic group of prime order \cite[Theorem 1.2]{CP}, and for exactly six nonabelian simple groups \cite[Theorem 1.1]{AIPT}, namely  the Janko groups $J_1, J_3$ and  $J_4$, together with  $Ly, Th$ and the Monster $M$. By \eqref{eq1}, $k(G)\geq3$ for all other finite simple groups. The problem of determining the closure number $k(G)$ for the remaining simple groups was posed in \cite[Question 3]{AIPT}, and in particular for determining all those with closure number $3$ (if any exist), see also \cite[Problem 3]{CP} and the Kourovka Notebook \cite[New Problems 20.2 and 20.3]{K}. 

A crucial tool for this work uses the notion  of a \emph{base} for a permutation group $G\leq \Sym(\Omega)$, that is to say, a set of points $\alpha_1,\dots,\alpha_{b}\in\Omega$ such that the only element of $G$ fixing each $\alpha_i$ is the identity. If $G\leq \Sym(\Omega)$ and $G$ has a base of size $b$, then  Wielandt \cite[Theorem 5.12]{W} proved that 
$G=G^{(b+1),\Omega}$. To see how this is used, let us consider the finite alternating groups $G=A_n$, $n\geq5$. For the natural action of $G$ on $\Omega=\{1,\dots,n\}$, and for $k\leq n-2$, the groups $G$ and $\Sym(\Omega)$ have exactly the same orbits on $\Omega^k$, and hence $G^{(k),\Omega}=\Sym(\Omega)\ne G$, so by definition the closure number satisfies $k(G)\geq n-1$.  On the other hand any $(n-2)$-subset of $\Omega$ is a base for $G$ and hence  $G^{(n-1),\Omega}= G$ by \cite[Theorem 5.12]{W}. We prove in Lemma~\ref{p:An} that, in fact, in every faithful permutation representation of $G$, the minimum size of a base is at most $n-2$, and hence by \cite[Theorem 5.12]{W}, the action is $(n-1)$-closed.  Thus the closure number is determined.

\begin{theorem}\label{t:An}
For $n\geq5$, the closure number $k(A_n)=n-1$. 
\end{theorem}
 
On the other hand we exploit known upper bounds on the base sizes of primitive permutation representations of the remaining simple groups in order to prove the following upper bounds 
for their closure numbers. 

\begin{theorem}\label{t:rest}
Suppose that $G$ is a finite nonabelian simple group that is not isomorphic to an alternating group. Then the following hold.
\begin{enumerate}
\item[\upshape{(a)}] If $G$ is a sporadic group, then $k(G)\leq 6$.
\item[\upshape{(b)}] If $G$ is an exceptional group of Lie type, then $k(G)\leqslant 7$.
\item[\upshape{(c)}] Suppose that $G$ is a classical group, with natural module of dimension $n$. Then either $G$ and an upper bound for $k(G)$ are given in Table \ref{tab:kclas}, or $k(G) \le \lfloor n/3+12 \rfloor$. In particular, if $n \ge 14$ in any case, then $k(G) \le n+2$.
\end{enumerate}
\end{theorem}

\begin{table}[ht]
\centering
\renewcommand{\arraystretch}{1.2}
\caption{Upper bounds for $k(G)$, for finite simple classical groups $G$ that satisfy the specified conditions, and are not isomorphic to alternating groups. Here, $\delta$ is the Kronecker delta.}
\label{tab:kclas}
\begin{tabular}{ ccc }
\hline
$G$ & Upper bound for $k(G)$ & Conditions \\
\hline
\hline
$\PSL(n,q)$ & $n+2-\delta_{2q}$ & Always\\
$\PSU(n,q)$ & $n+1$ & $n \in \{3,5\}$ or $n \ge 16$\\
$\PSp(n,q)$ & $n+1$ & $n \ge 14$\\
$\POm^\varepsilon(n,q)$ & $n$ & $n \ge 16$, with $q$ even if $n = 16$\\
\hline
\end{tabular}
\end{table}

We show in Lemma \ref{lem:M24} that $k(M_{24})=6$. Thus Theorem \ref{t:rest}(a) and \cite[Theorem 1.1]{AIPT} leave 19 sporadic groups $G$ for which $3\leqslant k(G)\leqslant 6$ and where we do not know the exact value.

\begin{question}
Determine $k(G)$ for these remaining $19$ sporadic simple groups.
\end{question}

The bound $k(G)\leq 7$ for exceptional groups of Lie type is derived (see Corollary~\ref{c:redn}) by using the fact that the minimum base size of primitive actions of such groups is at most 6. By \cite{B2018}, equality only holds for certain actions of $E_6(q)$ and $E_7(q)$, and so these are the only possibilities to have $k(G)=7$. However, we do not know if this bound is met.

\begin{question}
Determine a tight upper bound for $k(G)$ for an exceptional group $G$ of Lie type.
\end{question}

In many of the low-dimensional classical cases not specified here, our work in fact yields a smaller upper bound than the general bound of $\lfloor n/3+12 \rfloor$; see the comment following the proof of Theorem~\ref{t:rest}(c), before Corollary~\ref{c:exactprim}.
For simple $n$-dimensional classical groups we do not know whether the closure numbers grow linearly with $n$ or are in fact asymptotically smaller. 

\begin{question}
For a simple $n$-dimensional classical group $G$, is $k(G) = o(n)$?
\end{question}

In particular, deciding whether or not the closure numbers of finite simple classical groups are bounded or not, as the dimension of their natural module increases, would answer the open question \cite[20.3]{K} in the Kourovka Notebook.

For a group $G$ we define $k_{\mathrm{trans}}(G)$ to be the smallest integer $k$ such that $G^{(k),\Omega}=G$ for each set $\Omega$ upon which $G$ acts faithfully and transitively. Similarly, we define $k_{\mathrm{prim}}(G)$  to be the smallest integer $k$ such that $G^{(k),\Omega}=G$ for each set $\Omega$ upon which $G$ acts faithfully and primitively. Clearly, \[k(G)\geqslant k_{\mathrm{trans}}(G)\geqslant k_{\mathrm{prim}}(G).\]
We note that we can use the results of \cite{LPS1988} and \cite{PSaxl} to give absolute upper bounds on $k_{\mathrm{prim}}(G)$. In particular $k_{\mathrm{prim}}(G)\leqslant 6$ for all complete simple groups $G$.

\subsection*{Acknowledgements}
The research for this paper began at the 2022 Research Retreat of the Centre for the Mathematics of Symmetry and Computation of the University of Western Australia. The authors acknowledge discussions with their colleagues Vishnuram Arumugam and Alice Devillers.
The first author was supported by a St Leonard's International Doctoral Fees Scholarship and a School of Mathematics \& Statistics PhD Funding Scholarship at the University of St Andrews. The research of the second and third authors was supported by ARC Discovery Project Grants DP190101024 and DP190100450, respectively.

 \section{Preliminary results about $k$-closures}

We begin by noting the following result of Wielandt which we will often use without reference.

\begin{lemma}\label{lem:obvious} \cite[Theorem 5.6]{W}
Let $G$ act on a set $\Omega$ and $k\geqslant 1$ be an integer. Then  $h\in G^{(k),\Omega}$ if and only if for all $\alpha_1,\ldots,\alpha_k\in\Omega$ there exists $g\in G$ such that $(\alpha_1,\ldots,\alpha_k)^g=(\alpha_1,\ldots,\alpha_k)^h$.
\end{lemma}

Note that, because of our focus on the set of all faithful permutation representations of a given group $G$, we never write $G^{(k)}$ for a $k$-closure without giving also the set acted upon. So, for example, if $G$ acts on a set $\Omega$ and $\Delta$ is a $G$-invariant subset of $\Omega$, then  $(G^{(k)})^\Delta$ has no meaning in our work. We can write $G^{(k),\Delta}$ for the $k$-closure of the action of $G$ on $\Delta$,  and $(G^{(k),\Omega})^\Delta$ for the permutation group induced on $\Delta$ by the  $k$-closure of the action of $G$ on $\Omega$.

We also note the following result.

\begin{lemma}\label{lem:induced}
Let $G$ act on a set $\Omega$ and let $\Delta$ be a $G$-invariant subset of $\Omega$. Then $(G^{(k),\Omega})^\Delta \leqslant G^{(k),\Delta}$.
\end{lemma}
\begin{proof}
  \cite[Exercise 5.25]{W}.
\end{proof}

Next we prove a result about the $k$-closure of an intransitive action of a nonabelian simple group.



  \begin{lemma}\label{lem:simpleintrans}
  Let $G$ be a nonabelian simple group acting on a set $\Omega$ with orbits $\Delta_1,\ldots,\Delta_t$ with $t\geqslant2$. Suppose that for each $i,j\in\{1,\ldots, t\}$ there is a point $\beta\in\Delta_i$ such that $G_\beta$ is not transitive on $\Delta_j$, and also that $G^{(k),\Delta_i}=G^{\Delta_i}\cong G$ for each $i=1,\ldots, t$ and some $k\geq 2$. Then $G^{(k),\Omega}=G^\Omega\cong G$.
  \end{lemma}
  \begin{proof}
  For each $i=1,\ldots, t$, $G$ acts faithfully on $\Delta_i$ since $G^{(k),\Delta_i}\cong G$. Also, by Lemma~\ref{lem:induced}, $G^{(k),\Omega} \leqslant    \prod_{i=1}^t (G^{(k),\Omega})^{\Delta_i} \leqslant \prod_{i=1}^t G^{(k),\Delta_i} \cong G^t$. It follows that $G^{(k),\Omega}$ is a subdirect subgroup of $G^{(k),\Delta_1}\times\cdots\times G^{(k),\Delta_t}\cong G^t$, and so,  by a lemma of Scott (\cite[p.\,328]{scott} or see \cite[Theorem 4.16]{PS}), the group $G^{(k),\Omega}$ is equal to a direct product of $\ell$ diagonal subgroups $H_j\cong G$ for $j=1,\ldots, \ell$ (where $\ell\geq 1$), with pairwise disjoint supports, that is to say, there is a partition $\{\mathcal{O}_1,\ldots,\mathcal{O}_\ell\}$ of $\{\Delta_1,\ldots,\Delta_t\}$ such that, for each $j$, $(H_j)^\Delta\cong G$ for all $\Delta\in\mathcal{O}_j$ and $(H_j)^\Delta=1$ for all $\Delta\notin\mathcal{O}_j$. In particular, $G^{(k),\Omega}\cong G^\ell$. 
    Suppose that $\ell >1$. Choose $\Delta_i\in\mathcal{O}_1$ and $\Delta_j\in\mathcal{O}_2$.  Then the kernel $(G^{(k),\Omega})_{(\Delta_i)}$ of the action on $\Delta_i$ contains the diagonal subgroup $H_2$, which induces $G$ on $\Delta_j$. By assumption, there exists $\beta\in \Delta_i$ such that $G_\beta$ is not transitive on $\Delta_j$. Choose  $\alpha\in\Delta_j$ and consider the $k$-tuple $(\alpha,\beta,\beta,\ldots,\beta)$ of points. By definition of the $k$-closure, for each $h\in H_2 < G^{(k),\Omega}$, there exists $g\in G$ such that $(\alpha,\beta,\beta,\ldots,\beta)^g=(\alpha,\beta,\ldots,\beta)^h=(\alpha^h,\beta,\beta,\ldots,\beta)$. In particular $\alpha^g=\alpha^h$ and $g\in G_\beta$. Since $H_2$ is transitive on $\Delta_j$ it follows that $G_\beta$ is transitive on $\Delta_j$, which is a contradiction. Thus $\ell=1$, and so  $G^{(k),\Omega}=G^\Omega\cong G$.
  \end{proof}

 \begin{corollary}\label{cor:alltransequiv}
   Let $G$ be a nonabelian simple group acting on a set $\Omega$ with orbits $\Delta_1,\ldots,\Delta_t$ with $t\geqslant2$ such that the actions of $G$ on each orbit are pairwise equivalent. If, for some $k$, $G^{(k),\Delta_i}=G^{\Delta_i}\cong G$, then $k\geq2$ and $G^{(k),\Omega}=G^\Omega \cong G$.
 \end{corollary}
 
 \begin{proof} Since the actions of $G$ are pairwise equivalent, if $\beta\in \Delta_i$ then $G_\beta$ fixes a point in each orbit and so is not transitive on any $\Delta_j$. Suppose that $G^{(k),\Delta_i}=G^{\Delta_i}\cong G$, for some $k$. Then $k\geq2$, since $G^{(1),\Delta_i}={\rm Sym}(\Delta_i) > G^{\Delta_i}$ as $G^{\Delta_i}\cong G$ is a nonabelian simple group.  The assertion now follows from Lemma~\ref{lem:simpleintrans}.
 \end{proof}

  Next we need some results about the $k$-closure of an imprimitive group.

     \begin{lemma}\label{lem:onblocks}\label{lem:GBB}\label{l:kernel}
  Let $G$ be a group with a transitive action on a set $\Omega$, and suppose that $\mathcal{B}$ is a nontrivial $G$-invariant partition of $\Omega$. Let $k$ be an integer such that $2\leqslant k\leqslant |\mathcal{B}|$ and let $B\in\mathcal{B}$. Then the following hold:
  \begin{enumerate}
      \item[\upshape{(a)}] $\mathcal{B}$ is a $G^{(k),\Omega}$-invariant partition.
            \item[\upshape{(b)}]  $(G^{(k),\Omega})^{\mathcal{B}}\leqslant G^{(k),\mathcal{B}}$.
            \item[\upshape{(c)}] $(G^{(k),\Omega})_B^B\leqslant (G_B)^{(k),B}$.
             \item[\upshape{(d)}] If there exist $B_1,\ldots,B_k\in\mathcal{B}$ such that $G_{B_1,\ldots,B_k}=1$, then $G^{(k),\Omega}$ acts faithfully on $\mathcal{B}$.
       \end{enumerate}
    \end{lemma}
 
  \begin{proof}
  Let $g\in G^{(k),\Omega}$.
  Let $B\in\mathcal{B}$ and choose $\alpha,\beta\in B$ with $\alpha\neq \beta$, and consider the $k$-tuple $(\alpha,\beta,\beta,\ldots,\beta)$. Since $g\in G^{(k),\Omega}$, there exists $h\in G$ such that $(\alpha,\beta,\beta,\ldots,\beta)^h=(\alpha,\beta,\beta,\ldots,\beta)^g$. Then since $\mathcal{B}$ is a $G$-invariant partition of $\Omega$, it follows that $\alpha^g$ and $\beta^g$ lie in the same part of $\mathcal{B}$. Hence  $\mathcal{B}$ is also invariant under $G^{(k),\Omega}$ and part (a) holds

Now consider $g^{\mathcal{B}}$. Let $B_1,\ldots, B_k\in\mathcal{B}$ and choose $\beta_i\in B_i$, for $i=1,\dots,k$. Then there exists $h\in G$ such that $(\beta_1,\ldots,\beta_k)^g= (\beta_1,\ldots,\beta_k)^h$. Thus $(B_1,\ldots,B_k)^{g^\mathcal{B}}=(B_1,\ldots,B_k)^{h^{\mathcal{B}}}$ and so $g^{\mathcal{B}}\in G^{(k),\mathcal{B}}$, as required for part (b).

 Next, let $h\in (G^{(k),\Omega})_B$, so that $h^B\in (G^{(k),\Omega})_B^B$. To show that $h^B\in (G_B)^{(k),B}$ we prove the criteria of Lemma \ref{lem:obvious}, namely for all $k$-tuples $(\alpha_1,\dots,\alpha_k)\in B^k$, we find an element $x\in (G_B)^{B}$ such that $(\alpha_1,\dots,\alpha_k)^{h^B}=(\alpha_1,\dots,\alpha_k)^x$. 
  Since $h\in (G^{(k),\Omega})_B$, there exists $y\in G$ such that  $(\alpha_1,\dots,\alpha_k)^{h}=(\alpha_1,\dots,\alpha_k)^y$. Further, since $h$ fixes $B$ setwise it follows that $y\in G_B$. Then taking $x:=y^B$ we have the required equality for (c).
  
   Finally, suppose that there exist $B_1,\ldots,B_k\in\mathcal{B}$ such that $G_{B_1,\ldots,B_k}=1$. Let $h$ lie in the kernel of the action of $G^{(k),\Omega}$ on $\mathcal{B}$, and choose $\alpha\in\Omega$. Since $G$ acts transitively on $\Omega$ there exists $g\in G$ such that $\alpha^{g^{-1}}\in B_1$. For each $i=2,3,\ldots,k$ choose $\beta_i\in B_i^g$. Since $h\in G^{(k),\Omega}$, there exists $x\in G$ such that $(\alpha,\beta_2,\ldots,\beta_k)^h=(\alpha,\beta_2,\ldots,\beta_k)^x$. By its definition, $h$ fixes each $B_i^g$ setwise, and so $\beta_i^x\in B_i^g$ for all $i$ and $\alpha^x\in B_1^g$. Thus $x\in G_{B_1^g,B_2^g,\ldots,B_k^g}=(G_{B_1,\ldots,B_k})^g=1$. Hence $\alpha^h=\alpha^x=\alpha$. Since $\alpha$ was arbitrary and $G^{(k),\Omega}\leqslant\Sym(\Omega)$, it follows that $h=1$ and so $G^{(k),\Omega}$ acts faithfully on $\mathcal{B}$ as stated in part~(d).
  \end{proof}
  
   


  


\section{Links between bases and closures}\label{s:bases}

For a finite group $G$ acting on a set $\Omega$, let $b(G^\Omega)$ denote the minimum size of a base for the permutation group $G^\Omega \leq\Sym(\Omega)$ induced by $G$ on $\Omega$; $b(G^\Omega)$ is called the \emph{base size} of $G^\Omega$. When $G$ is transitive, a \emph{maximal block system} for $G$ on $\Omega$ is a $G$-invariant partition $\Sigma$ of $\Omega$ such that $|\Sigma|>1$ and the induced action $G^\Sigma$ is primitive.  In other words, the only $G$-invariant partition strictly coarser than $\Sigma$ (that is, with each part a union of more than one part of $\Sigma$) is the universal partition $\{\Omega\}$ with just one part.  Note that $\Sigma$ is allowed to be the partition into singletons.

\begin{lemma}\label{l:redn}
Let $G$ be a finite nonabelian simple group, and suppose that $G$ has a faithful permutation representation on a set $\Omega$. Let $\Omega_0$ be a $G$-orbit on $\Omega$ with $|\Omega_0|>1$, and let $\Sigma$ be a maximal block system for $G$ on $\Omega_0$. Then the base size $b(G^\Omega)$ satisfies  $b(G^\Omega)\leq b(G^\Sigma)$.   
\end{lemma}

\begin{proof}
Since $|\Sigma|>1$ and $G$ is simple it follows that $G$ acts faithfully on $\Sigma$. Let $b=b(G^\Sigma)$. Then $G$ has a base $\{\sigma_1,\dots, \sigma_b\}$ in $\Sigma$, that is $G_{\sigma_1,\dots, \sigma_b}=1$. Let $\alpha_i\in\sigma_i$ for each $i$. Then $G_{\alpha_1,\dots, \alpha_b}$ fixes each of the $\sigma_i$ setwise, and hence 
$G_{\alpha_1,\dots, \alpha_b}\leq G_{\sigma_1,\dots, \sigma_b}=1$. It follows that $b(G^\Omega)\leq b$.  
\end{proof}

We may restate Wielandt's result \cite[Theorem 5.12]{W} as follows.

\begin{lemma}\label{l:W}
If $G$ is a finite group, and $G$ has a permutation representation on a set $\Omega$, such that $b(G^\Omega)=b$, then $G^{(b+1),\Omega}=G$.        
\end{lemma}

This has the following useful consequence.

\begin{corollary}\label{c:redn}
Let $G$ be a finite nonabelian simple group, and let $b_{maxprim}(G)$ denote the maximum of $b(G^\Omega)$ over each faithful primitive permutation representation of $G$ on a set $\Omega$. Then the closure number $k(G)$ is at most $b_{maxprim}(G)+1$. 
\end{corollary}

\begin{proof}
Let $G$ be faithfully represented on a set $\Omega$. Then, by Lemma~\ref{l:redn}, $b(G^\Omega)$ is at most the base size of some primitive representation of $G$, and hence  $b(G^\Omega)\leq b_{maxprim}(G)=b$, say. It follows from Lemma~\ref{l:W} that $G^{(b+1),\Omega}=G$. Hence, by  definition of the closure number, $k(G)\leq b+1$. 
\end{proof}

In the light of Corollary~\ref{c:redn}, we next summarise some recent results on base sizes of finite primitive permutation groups that are nonabelian simple groups. For several families of such actions  the base size is much larger than all others; these are called standard actions and are defined as follows. 

\begin{definition}\label{d:std}
Let $G$ be a finite nonabelian simple group. A primitive action of $G$ on a set $\Omega$
is \emph{standard} if, up to equivalence of actions, one of (a) or (b) holds for $G^\Omega$, and the action is non-standard otherwise.
\begin{enumerate}
\item[\upshape{(a)}] $G = A_n$ for some $n\geq5$, and $\Omega$ is an orbit of subsets or 
partitions of $\{1, \dots , n\}$; or

\item[\upshape{(b)}] $G$ is a classical simple group with natural
module $V$ over a field of characteristic $p$, so $G=\PSL(V), \PSp(V), \PSU(V)$ or $\POm^\varepsilon(V)$, and $G$ has a \emph{subspace action} on $\Omega$, that is to say, one of the following holds:
	\begin{enumerate}
	\item[\upshape{(i)}] for some positive integer $k \le \dim(V)/2$, $\Omega$ is a $G$-orbit of $k$-dimensional totally singular, nondegenerate, or (if G is orthogonal, $p = 2$ and $k = 1$) nonsingular subspaces of $V$; or
	\item[\upshape{(ii)}] $G = \PSp(V)$, $p=2$, and $\Omega$ is the set of nondegenerate quadratic forms on $V$ of fixed type $+$ or $-$ that polarise to the symplectic form preserved by $G$. 
	\end{enumerate}
\end{enumerate}
\end{definition}

If $G = \PSL(V)$ in case (b)(i) above, then $\Omega$ consists of all $k$-dimensional subspaces of $V$. Note also that, as the group $G$ is simple, the definition above is slightly simpler than that for a standard action of an almost simple group as given by Cameron and Kantor\footnote{In some cases (e.g., \cite[Definition 2.1]{B}), subspace actions are defined even more generally, allowing for certain additional actions of (non-simple) almost simple groups with socle $\PSp(4,q)'$ ($q$ even) or $\POm^+(8,q)$; see the proof of \cite[Proposition 8]{MR}.}in  \cite{CK}, where they conjectured that for each non-standard primitive action of an almost simple group $G$ on a finite set $\Omega$, the base size $b(G^\Omega)$ is bounded above by an absolute constant $c$.  
In \cite[Theorem 1.3]{LS}, Liebeck and Shalev proved this conjecture, but did not specify the constant $c$. In the case of a  non-standard primitive group $G$ with socle $A_n$, a sketch of a proof was given in \cite[Proposition 2.3]{CK} that $c=2$ should usually suffice. A detailed proof in \cite[Theorem 1.1 and Corollary 1.2]{BGS} showed that this is indeed true for all $n\geq 13$, and all possibilities for non-standard actions of $A_n$ and $S_n$ with base-size at least $3$ were given in \cite[Table 1]{BGS} (although the group $A_6$ appears in this table, the corresponding action is standard). Burness and others \cite{B,BLS,BOW} determined the best possible value for $c$ such that $b(G^\Omega)\leq c$ for all other non-standard primitive actions of almost simple groups $G$. In particular, upper bounds for $b(G^\Omega)$ were specified in \cite[Theorem 1.1, Proposition 4.1]{B} in the case where $G^\Omega$ is a non-standard action of an almost simple classical group. Thus we obtain the following.

\begin{lemma}\label{l:bs1}
Let $G$ be a finite nonabelian simple group, and suppose that $G$ has a non-standard primitive action on a finite set $\Omega$ with $|\Omega|>1$. Then the following statements hold.
\begin{enumerate}
\item[\upshape{(i)}] $b(G^\Omega)\leq 7$, with equality if and only if $G=M_{24}$ in its $5$-transitive action of degree $24$.
\item[\upshape{(ii)}] If $G=A_n$ for some $n\geq5$, then either $b(G^\Omega)< n-2$ or $(n, |\Omega|, b(G^\Omega))=(5,6,3)$.
\item[\upshape{(iii)}] If $G$ is a classical group, then $b(G^\Omega) \le 4$.
\item[\upshape{(iv)}] If $b(G^\Omega) = 4$ and $G$ is equal to $\PSL(n,q)$, $\PSU(n,q)$ or $\PSp(n,q)$ with $n \le 5$, then $G = \PSU(4,3)$.
\end{enumerate}
\end{lemma}

We are now in a position to prove part of Theorem \ref{t:rest}.

\medskip
\noindent \emph{Proof of Theorem \ref{t:rest}(b).}
 Let $G$ be an exceptional group of Lie type. Then all primitive actions of $G$ are non-standard and so by Lemma \ref{l:bs1}, $b_{maxprim}(G)\leqslant 6$. Hence Corollary \ref{c:redn} implies that $k(G)\leqslant 7$, as required.
 \hfill\qedsymbol




\subsection{Bases for alternating groups}
For the actions of $A_n$ or $S_n$ on $k$-subsets, the best general bounds have been obtained by Halasi in \cite{H}. To extract the information we need from his work it is simplest to use his monotonicity result from \cite[Section 2]{H}.

\begin{lemma}\label{l:ksets}
Let $G=A_n$ or $S_n$, where $n\geq5$ and $n=3m+\delta$ with $\delta\in\{0,1,2\}$. Consider the action of $G$ on the set $\Omega$ of $k$-subsets of $X=\{1,2,\dots, n\}$, where $2\leq k\leq n/2$. Then $b(G^\Omega)\leq 2m$ if $\delta\in\{0,1\}$ and $b(G^\Omega)\leq 2m+1$ if $\delta=2$; so in particular $b(G^\Omega)\leq 2n/3$. Moreover, $b(G^\Omega)\leq n-2$, with equality if and only if 
$(n,k, G)=(6,2, S_6)$ or $(5,2, S_5)$. 
\end{lemma}
 
\begin{proof}
Since $b(A_n^\Omega)\leq b(S_n^\Omega)$, we may assume that $G=S_n$ for proving the main assertion.
As in \cite{H}, let $f(n,k)$ denote the (minimum) base size of the action of $S_n$ on the set of $k$-subsets of $X$. By \cite[Corollary 2.2]{H}, $f(n,k)\leq f(n,2)$, so we may assume further that $k=2$. Write $n=3m+\delta$ as in the statement, and  consider the following subset of $\Omega$:
\[
B=\{\{ 3j+1, 3j+2\}, \{ 3j+2, 3j+3\} \mid 0\leq j\leq m-1\}.
\]
Let $H$ be the pointwise stabiliser of $B$ in $G$. Then $H$ fixes each point $i$ of $X$ such that $1\leq i\leq 3m$. Suppose first that $\delta=0$ or $1$ (so that $n \ge 6$). Then $H=1$ and hence $b(G^\Omega)\leq |B|=2m\leq 2n/3\leq n-2$. Further if $b(G^\Omega)= n-2$ then $2n/3 = n-2$, 
and hence $n=6$ and $b(G^\Omega)= 4$.  By \cite[Theorem 3.1]{H}, $f(6,3)=\lceil \log_2 6\rceil = 3$, which is less than  $b(G^\Omega)= 4$, and hence $k=2$. Also, by   
\cite[Theorem 3.2]{H}, $f(6,2)\geq \lceil (2n-2)/(k+1)\rceil = 4$, and it follows that $b(S_6^\Omega)= 4$. However $b(A_6^\Omega)= 3$ since the pairs $\{1,2\}, \{1,3\}, \{1,4\}$ form a base. Thus the result holds if $\delta\leq 1$.

 Suppose now that $\delta=2$. Then $H = \langle (n-1, n)\rangle$, and the pointwise stabiliser of $B\cup\{ \{1, n\}\}$ is trivial. Hence  $b(G^\Omega)\leq |B|+1=2m+1$ and we have both $2m+1<2n/3$ and  $2m+1\leq 3m = n-2$. In this case, if  $b(G^\Omega)= n-2$, then $2m+1 = 3m$, 
 so $n=5$. 
The proof is complete on noting that $b(S_5^\Omega)= 3$ while $b(A_5^\Omega)= 2$.
\end{proof}

Similarly for the base sizes of the actions of $A_n$ or $S_n$ on partitions,  precise results have been obtained by Burness, Garonzi and Lucchini in \cite{BGL} for small part sizes, building on the work of Benbenishty, Cohen and Niemeyer \cite{BCN},  and by Morris and Spiga in \cite{MS} for large part sizes. Again, the information we require is most easily obtained  for most cases from earlier work, this time \cite{BGL} and an argument in  Liebeck's 1984 paper \cite{L}. 

\begin{lemma}\label{l:partns}
Let $G=A_n$ or $S_n$, where $n\geq5$ and $n=ab$ with $a\geq2, b\geq2$. Consider the action of $G$ on the set $\Omega$ of partitions of $X=\{1,2,\dots, n\}$ with $b$ parts of size $a$.  Then $b(G^\Omega)\leq n-2$ with equality if and only if 
$(a,b, G)=(2,3, S_6)$ or $(a,b,G)=(3,2,S_6)$.
\end{lemma}

\begin{proof}
Since $b(A_n^\Omega)\leq b(S_n^\Omega)$, we may assume that $G=S_n$ for proving the main assertion. If $a\geq3$ then in \cite[Case (i) on pp. 13--14]{L}, Liebeck proved that $b(G^\Omega)\leq (a-1)(b-1)+2=ab-a-b+3\leqslant n-2$,  where the second equality is strict unless $(a,b)=(3,2)$. In this case we have that \cite[Theorems 1.1 and 1.2]{MS} $b(S_n^\Omega)=4=n-2$ while $b(A_n^\Omega)=3<n-2$.  
Thus we may assume that $a=2$, so $b\geq3$. In this case, by \cite[Theorem 2]{BGL}, the base size satisfies
$b(S_n^\Omega)=3<  n-2$ if $b\geq4$, and $b(S_n^\Omega)=4=  n-2$ if $b=3$.
Further,  $b(A_n^\Omega)=3<  n-2$ if $b=3$, see \cite[Remark 2.8]{BGL}.
\end{proof}

Finally, we prove the following lemma, which completes the proof of Theorem \ref{t:An}.

\begin{lemma}\label{p:An}
Let $G=A_n$ where $n\geq5$, and suppose that $G$ has a faithful primitive permutation representation on a finite set $\Omega$. Then $b(G^\Omega)\leq n-2$, with equality if and only if either $|\Omega|=n$ or $(n, |\Omega|, b(G^\Omega))=(5,6,3)$.  Moreover the assertions of Theorem~\ref{t:An} are valid.
\end{lemma}

  \begin{proof}
If $|\Omega|=n$ then the $G$-action on $\Omega$ is permutationally isomorphic to its natural action and $b(G^\Omega)=n-2$. Suppose this is not the case.
If the primitive permutation group $G^\Omega$ is standard, as in Definition~\ref{d:std}(a), then by Lemmas~\ref{l:ksets} and~\ref{l:partns}, 
$b(G^\Omega)<n-2$. On the other hand if $G^\Omega$ is non-standard, then by Lemma~\ref{l:bs1}, either $b(G^\Omega)< n-2$ or $(n, |\Omega|, b(G^\Omega))=(5,6,3)$. This proves the first assertion.

Now fix $n\geq5$, and let $\Omega$ be such that $G=A_n$ has a faithful primitive action on $\Omega$ and $b(G^\Omega)$ is maximal over all such actions, that is, $b(G^\Omega) = b_{maxprim}(G)$. By Corollary~\ref{c:redn}, the closure number $k(G)$ satisfies $k(G)\leq b(G^\Omega)+1$, and we have just proved that $b(G^\Omega)\leq n-2$. Thus $k(G)\leq n-1$. On the other hand, as we noted in the introduction, for the natural action of $G$ on $\Omega=\{1,\dots,n\}$,  $G$ and $S_n$ have the same orbits on $\Omega^{n-2}$, and hence $G^{(n-2),\Omega}=S_n\ne G$. Thus $k(G)>n-2$, and hence $k(G)=n-1$.  This proves Theorem~\ref{t:An}.
  \end{proof}

\section{Bases for simple classical groups}
\label{subsec:classical}

In this subsection, we consider in detail base sizes of finite simple classical groups $G$ that are not isomorphic to alternating groups. We are concerned with
\begin{equation}\label{e:bmaxprim}
    \mbox{  $b_{maxprim}(G):=\max\{b(G^\Omega) \mid G \text{ acts faithfully and primitively on the set } \Omega\}$.}
\end{equation}
In particular, we prove the following theorem, which gives lower and upper bounds for  $b_{maxprim}(G)$.
 Together with Corollary~\ref{c:redn}, these bounds allow us to prove Theorem~\ref{t:rest}(c). The bounds depend on certain quantities which we define as follows.
 
\begin{equation}\label{e:O}
\mathcal{O}:=\{ \mathrm{P}\Om^\varepsilon(n,q)\mid \mbox{$n$ even, $\varepsilon=\pm$, $q \le 3$, and if $q = 3$, then $n\equiv 1-\varepsilon.1\pmod{4}\}$}
\end{equation} 
 
\begin{equation}\label{e:abc}
 \alpha:=\left\{\begin{array}{ll}
      1&\mbox{if $(q+1) \mid n$}  \\
      0&\mbox{otherwise,} 
 \end{array}\right. 
 \ 
 \beta:=\left\{\begin{array}{ll}
      1&\mbox{if $G\in\mathcal{O}$}  \\
      0&\mbox{otherwise,} 
 \end{array}\right. 
 \ 
  \gamma:=\left\{\begin{array}{ll}
      1&\mbox{if $q$ even, $n \in \{10,12,14,16\}$}  \\
      0&\mbox{otherwise.} 
 \end{array}\right. 
\end{equation}

\begin{table}[ht]
\centering
\renewcommand{\arraystretch}{1.2}
\caption{Lower and upper bounds for $b_{maxprim}(G)$, for finite simple classical groups $G$ not isomorphic to alternating groups, with $\alpha,\beta,\gamma$ as in \eqref{e:abc}, and $\delta$ the Kronecker delta.}
\label{tab:maxprim}
\begin{tabular}{ ccc }
\hline
$G$ & Lower bound & Upper bound \\
\hline
\hline
$\PSL(n,q)$ & $n+1-\delta_{2q}$ & $n+1-\delta_{2q}$\\
\hline
$\PSU(4,q)$ & $4-\alpha$ & $5$\\
$\PSU(6,q)$ & $6-\alpha$ & $12$\\
$\PSU(n,q)$, $7 \le n \le 15$ & $n-\alpha$ & $\lfloor n/3 + 11 \rfloor$\\
$\PSU(n,q)$, $n \in \{3,5\}$ or $n \ge 16$ & $n-\alpha$ & $n$\\
\hline
$\PSp(4,q)$ & $4$ & $5$\\
$\PSp(n,q)$, $6 \le n \le 12$ & $n$ & $\lfloor n/3 + 10 \rfloor$\\
$\PSp(n,q)$, $n \ge 14$ & $n$ & $n$\\
\hline
$\POm^\varepsilon(n,q)$, $7 \le n \le 16$ & $n-1-\beta$ & $\lfloor n/3 + 11-\gamma \rfloor$\\
$\POm^\varepsilon(n,q)$, $n \ge 17$ & $n-1-\beta$ & $n-1$\\
\hline
\end{tabular}
\end{table}


\begin{theorem}
\label{t:bsub}
Let $G$ be a finite simple classical group that is not isomorphic to an alternating group. Then Table~\ref{tab:maxprim} gives a lower and upper bound for $b_{maxprim}(G)$.
\end{theorem}

In order to prove this theorem, it will be necessary to consider primitive subspace actions $G^\Omega$, as in Definition~\ref{d:std}. The next result, Lemma~\ref{l:sublarge}, summarises results from \cite{HLM}.  Note that if $\Omega$ is a $G$-orbit of $k$-dimensional subspaces of the natural module $V$ for $G$, with $n:=\dim(V)$, then (up to equivalence of actions) $1 \le k \le n/2$. In fact, if $G = \POm^-(n,q)$, then $n/2-1$ is the maximum dimension of a totally singular subspace of $V$. On the other hand, if $G \ne \POm^-(n,q)$ or if $4 \nmid n$, then $G$ does not act primitively on any set of $n/2$-dimensional nondegenerate subspaces of $V$ (see \cite[Table 2.2]{BHR}).

\begin{lemma}
\label{l:sublarge}
Let $G$ be a finite simple  classical group with natural module $V$ of dimension $n$. Additionally, let $\Omega$ be a set of $k$-dimensional nondegenerate or totally singular subspaces of $V$, such that $G^\Omega$ is a primitive subspace action. Then $b(G^\Omega) \le n/k+11$. Furthermore, $b(G^\Omega) \le n/k+10$ if $G = \PSp(n,q)$ and $k$ is odd; if $G = \POm^\pm(n,q)$, $k$ is odd and $q$ is even; or if $k = n/2$ and either $G \ne \POm^-(n,q)$ or $4 \nmid n$.
\end{lemma}

\begin{proof}
Theorem 3.3 of \cite{HLM} shows that $b(G^\Omega) \le n/k+11$. If $k = n/2$ and either $G \ne \POm^-(n,q)$ or $4 \nmid n$, then $\Omega$ is a set of totally singular subspaces of $V$ (since $G^\Omega$ is primitive). This is also true if $k$ is odd and either $G = \PSp(n,q)$, or $G = \POm^\pm(n,q)$ with $q$ even, as here no nondegenerate subspace of $V$ has odd dimension. Thus in any of these special cases, $b(G^\Omega) \le n/k+10$ by \cite[p.~28]{HLM}.
\end{proof}

It would be possible to use Lemmas~\ref{l:bs1} and \ref{l:sublarge} directly (along with some additional arguments) to obtain upper bounds for $b_{maxprim}(G)$, at least in the case where $G$ has no primitive action on (degenerate) nonsingular $1$-dimensional subspaces of $V$. However, in most cases, we can deduce significantly improved bounds by considering in detail the known bounds on base sizes of primitive actions of $G$ on subspaces of dimension at most $2$.

By case (b) of Definition~\ref{d:std}, any subspace action of $G$ is defined with respect to the natural module for $G$. Thus $G$ may have standard actions that are not subspace actions of $G$, but are equivalent to subspace actions of isomorphic classical groups. Therefore, in the lemma below, it is not sufficient to consider $G$ only up to isomorphism. For example, the orthogonal groups $G = \POm^\varepsilon(n,q)$ (with $\varepsilon = -$ if $n = 4$ so that $G$ is simple) occur as classical simple groups for any integer $n\geq 3$,  even though these groups are often considered as linear, unitary or symplectic groups of lower dimension when $n \le 6$.

\begin{lemma}
\label{l:lowtab}
Let $G$ be a finite simple classical group that is not isomorphic to an alternating group, and let $V$ be the natural module for $G$. Additionally, let $\Omega$ be a set of $k$-dimensional subspaces of $V$, such that $k \le 2$ and $G^\Omega$ is a primitive subspace action. Then Table~\ref{tab:lowdim} gives an upper bound for $b(G^\Omega)$.
\end{lemma}

\begin{table}[ht]
\centering
\renewcommand{\arraystretch}{1.2}
\caption{Upper bounds for the base sizes of primitive subspace actions in Lemma~\ref{l:lowtab}. Here, $\alpha:=1$ if $\varepsilon = -$, and $\alpha:=0$ otherwise. Additionally, $\delta$ is the Kronecker delta.}
\label{tab:lowdim}
\begin{tabular}{ ccc }
\hline
$G$ & $k$ & Upper bound for $b(G^\Omega)$\\
\hline
\hline
$\PSL(n,q)$ & $1$ & $n+1-\delta_{2q}$\\
$\PSL(n,q)$, $n \ge 4$ & $2$ & $\lceil n/2 \rceil + 2 + \delta_{4n}$\\
\hline
$\PSU(n,q)$ & $1$ & $n + \delta_{2n}$\\
$\PSU(n,q)$, $n \ge 4$ & $2$ & $\lceil n/2 \rceil + 3\delta_{4n}+\delta_{5n}+\delta_{6n}$\\
\hline
$\PSp(n,q)$ & $1$ & $n+\delta_{2n}$\\
$\PSp(n,q)$, $n \ge 4$ & $2$ & $\lceil n/2 \rceil + 2\delta_{4n} + \delta_{6n}$\\
\hline
$\POm^\varepsilon(n,q)$, $n \ge 3$ & $1$ & $n-1+\delta_{3n}+\delta_{5n}$\\
$\POm^\varepsilon(n,q)$, $n \ge 4$ & $2$ & $\lceil n/2 \rceil + \delta_{4n} + \delta_{5n} + \alpha \delta_{6n}$\\
\hline
\end{tabular}
\end{table}


\begin{proof}
Let $b(G,k)$ be the value in the final column of Table~\ref{tab:lowdim} corresponding to the pair $(G,k)$. We first consider the case where $G = \PSL(n,q)$ and $k = 1$. Let $\{e_1,\ldots,e_n\}$ be a basis for $V$. It is well known (and easy to verify) that $B:=\{\langle e_1 \rangle, \ldots, \langle e_n \rangle\}$ is a base of minimal size for $G^\Omega$ if $q = 2$, and that $B \cup \{\langle e_1+\cdots+e_n \rangle\}$ is a base of minimal size if $q > 2$. Hence $b(G^\Omega) = b(G,k)$.

Next, if $G \in \{\PSU(2,q),\PSp(2,q)\}$ with $k=1$, then the subspaces in $\Omega$ are $1$-dimensional and totally singular. Every $1$-dimensional subspace of $\PSp(2,q)$ is totally singular, and so the action in this case is precisely the unique primitive subspace action of $\PSL(2,q)$ (which is equal to $\PSp(2,q)$). This action is in fact also equivalent to the unique subspace action of $\PSU(2,q)$ (see \cite[Tables 2.3 \& 8.1]{BHR}). As $G$ is simple, $q \ne 2$ for these groups, and we deduce that $b(G^\Omega) = 3 = b(G,k)$.

Thus we may assume that $n \ge 3$, and that if $k = 1$ then $G$ is not linear. Additionally, if $G$ is orthogonal, we assume here that $n \ge 4$, that $k = 1$ if $n \le 6$, and that $\Omega$ is a set of nondegenerate or nonsingular subspaces if $n = 4$ (we deal with the other orthogonal cases in the next paragraph). The results in \cite[\S2.1--2.2]{MR} show that $b(G,k)$ is an upper bound for the base size of the action on $\Omega$ of the projective isometry group $\PGL(n,q),\PGU(n,q),\PSp(n,q)$ or $\PGO^\varepsilon(n,q)$ associated with and containing $G$. As $b(G^\Omega) \le b(R^\Omega)$ for any overgroup $R$ of $G$, it follows that $b(G^\Omega) \le b(G,k)$.



Finally we deal with the actions of small dimensional orthogonal groups not covered by the previous paragraph, namely the groups $G= \POm^\varepsilon(n,q) \cong H$, where $(n,\varepsilon,H)$ is one of the following: \[
(3,\circ,\PSL(2,q)),(4,-,\PSL(2,q^2)),(5,\circ,\PSp(4,q)),(6,+,\PSL(4,q)),(6,-,\PSU(4,q)).
\] 
We can use the tables in \cite[\S2.2, \S8.2]{BHR} to determine the maximal subgroup type of a point stabiliser of $G^\Omega$, considered as a subgroup of $H$. Using this information, we deduce from \cite[Tables 2--3]{B} that if $G^\Omega$ is not equivalent to a subspace action of $H$, then $b(G^\Omega) \le 3$, and in each case this is at most the relevant bound $b(G,k)$. We also see that if $G^\Omega$ is equivalent to a subspace action of $H$ in the case where $n = 6$ and $k = 2$, then $\varepsilon = -$ and the action of $H$ is on totally singular $1$-dimensional subspaces. In each case, the bounds on base sizes of subspace actions of $H$ from above show that $b(G^\Omega) \le b(G,k)$.
\end{proof}


We can now prove Theorem~\ref{t:bsub}. Since $b_{maxprim}(G)$ is a property of $G$ as an abstract group, we may now consider $G$ up to isomorphism.

\medskip\noindent
{\sc Proof of Theorem~\ref{t:bsub}}\quad
Let $V$ be the natural module for $G$ and $n:=\dim(V)$. We begin by determining a lower bound for $b_{maxprim}(G)$. For each positive integer $k \le n/2$, \cite{F} gives a lower bound for the base size of a primitive action of $G$ on a set of $k$-dimensional subspaces of $V$ (depending only on $G$ and $k$). Additionally, the tables in \cite[\S3.5]{KL} and \cite[\S8.2]{BHR} show that $G$ has a primitive action on a set of $1$-dimensional subspaces of $V$. In each case, the lower bound for $b_{maxprim}(G)$ given in Table~\ref{tab:maxprim} is the lower bound from \cite{F} in the case $k = 1$ (when $G$ is linear, this bound is well known; see the proof of Lemma~\ref{l:lowtab}).

Now, let $b_{maxsub}(G)$ and $b_{maxstand}(G)$ denote the maximum of the base sizes of all primitive subspace actions of $G$, and of all primitive standard actions of $G$, respectively. The latter accounts for all actions of $G$ that are equivalent to subspace actions of isomorphic classical groups. Additionally, let $b_{upper}(G)$ be the upper bound for $b_{maxprim}(G)$ from Table~\ref{tab:maxprim}. We will divide the remainder of the proof into several cases, corresponding to the type of $G$. In each case, we utilise the upper bounds from Lemma~\ref{l:lowtab} for base sizes of primitive subspace actions of $G$ on subspaces of $V$ of dimension at most two.

\medskip

\textbf{Case (a)}: $G \cong \PSL(n,q)$. Any primitive action of $G$ on $k$-dimensional subspaces of $V$, with $3 \le k \le n/2$, has base size at most $n/k+3$ if $k \mid n$, or at most $n/k+5$ otherwise \cite[p.~24]{HLM}. As $\lceil n/2 \rceil +2$ and $\lfloor n/3+5 \rfloor$ are less than or equal to $n$ for all $n \ge 7$, it follows easily from Lemma~\ref{l:lowtab} (noting that $G \not\cong \PSL(4,2) \cong A_8$) that $b_{maxsub}(G) \le n+1-\delta_{2q} = b_{upper}(G)$.

We now consider standard actions of $G$ that are not subspace actions. The relevant isomorphisms here are $\PSL(2,7) \cong \PSL(3,2)$, $\PSL(2,q) = \PSp(2,q) \cong \PSU(2,q) \cong \Om(3,q)$, $\PSL(2,q) \cong \POm^-(4,q^{1/2})$ (when $q$ is a square) and $\PSL(4,q) \cong \POm^+(6,q)$. The tables in \cite[\S2.2, \S8.2]{BHR} show that the action of $\POm^+(6,q)$ on $3$-dimensional totally singular subspaces is equivalent to a subspace action of $\PSL(4,q)$. It follows from Lemma~\ref{l:lowtab} that each subspace action of $\POm^+(6,q)$ or of $\PSL(4,q)$ has a base size of at most $5$, while the remaining actions on subspaces here all have a base size of at most $3$. When $q$ is even, the action of $\PSp(2,q)$ corresponding to case (b)(ii) of Definition~\ref{d:std} also has base size at most $3$ \cite[p.~30]{HLM}. Thus $b_{maxstand}(G) \le b_{upper}(G)$ in each case.

It remains to consider the cases where our upper bound $b_{upper}(G)$ for $b_{maxstand}(G)$ may be strictly less than $b_{maxprim}(G)$. By Lemma~\ref{l:bs1}, this can occur only if $n+1-\delta_{2q} < 4$, i.e., only if $n = 2$ or $G = \PSL(3,2)$. This same lemma now shows that the base size of a non-standard action of $G$ is at most $3 = n+1-\delta_{2q}$. Therefore, in each linear case, $b_{maxprim}(G) \le b_{upper}(G)$.

\medskip

\textbf{Case (b)}: $G \cong \PSU(n,q)$, with $n \ge 3$. As $\lfloor n/3+11 \rfloor \le n$ for all $n \ge 16$, we see from Lemmas~\ref{l:sublarge} and \ref{l:lowtab} that $b_{maxsub}(G) \le b_{upper}(G)$.

The remaining standard actions correspond to the isomorphisms $\PSU(4,q) \cong \POm^-(6,q)$ and $\PSU(4,2) \cong \PSp(4,3) \cong \POm(5,3)$. Lemma~\ref{l:lowtab} shows that $b_{maxstand}(G) \le b_{upper}(G)$ (since $4 \nmid 6$, the group $\POm^-(6,q)$ has no primitive subspace action on three-dimensional subspaces of its natural module). Moreover, if $b_{upper}(G) < 4$, then $b_{maxsub}(G) \le 3 = n$. Hence Lemma~\ref{l:bs1} yields $b_{maxprim}(G) \le b_{upper}(G)$ in each case.

\medskip

\textbf{Case (c)}: $G \cong \PSp(n,q)$, with $n \ge 4$. When $q$ is even, the subspace action of $G$ from case (b)(ii) of Definition~\ref{d:std} has base size $n$ if $n \ge 6$ \cite[Proposition 1]{MR}, or base size at most $5$ if $n = 4$ \cite[p.~30]{HLM}. Since $\lfloor n/3+10 \rfloor$ and $\lfloor n/4+11 \rfloor$ are less than or equal to $n$ for all $n \ge 14$, we obtain $b_{maxsub}(G) \le b_{upper}(G)$ from Lemmas~\ref{l:sublarge} and \ref{l:lowtab}.

We now consider the isomorphisms $\PSp(4,3) \cong \PSU(4,2) \cong \POm^-(6,2)$ and $\PSp(4,q) \cong \POm(5,q)$. In each case, Lemma~\ref{l:lowtab} yields $b_{maxstand}(G) \le b_{upper}(G)$. As no symplectic $G$ satisfies $b_{upper}(G) < 4$, it follows from Lemma~\ref{l:bs1} that $b_{maxprim}(G) \le b_{upper}(G)$.

\textbf{Case (d)}: $G \cong \POm^\varepsilon(n,q)$, with $n \ge 7$. Noting that $\lfloor n/3+11 \rfloor \le n-1$ for all $n \ge 17$ and $\lfloor n/4+11 \rfloor \le \lfloor n/3+10 \rfloor$ for all $n \ge 9$, we see from Lemmas~\ref{l:sublarge} and \ref{l:lowtab} that $b_{maxsub}(G) \le b_{upper}(G)$.

As $n \ge 7$, there are no nontrivial isomorphisms of classical groups to consider. Moreover, $b_{upper}(G) > 3$, and so $b_{maxprim}(G) \le b_{upper}(G)$ by Lemma~\ref{l:bs1}. \qed

\medskip

We can now prove Theorem \ref{t:rest} for classical groups.

\medskip
\noindent\emph{Proof of Theorem \ref{t:rest}(c).}
Let $G$ be a classical group that is not isomorphic to an alternating group. Then Theorem~\ref{t:rest}(c) follows from Corollary \ref{c:redn} and the upper bound for $b_{maxprim}(G)$ given in Theorem~\ref{t:bsub}.
\hfill\qedsymbol

\medskip
Of course, if $G$ is a low-dimensional unitary, symplectic or orthogonal group that is not explicitly mentioned in Theorem~\ref{t:rest}(c), then the general upper bound of $\lfloor n/3+12 \rfloor$ for $k(G)$ from this theorem is often higher than the actual upper bound that follows from Theorem~\ref{t:bsub}. Notice also that the lower and upper bounds for $b_{maxprim}(G)$ from Table~\ref{tab:maxprim} often coincide. The following corollary of Theorem~\ref{t:bsub} summarises when this occurs. Here, $\delta$ again denotes the Kronecker delta, and $\mathcal{O}$ is as in \eqref{e:O}.

\begin{corollary}
\label{c:exactprim}
Let $G$ be a finite simple classical group that is not isomorphic to an alternating group.
\begin{enumerate}
\item[\upshape{(i)}] If $G = \PSL(n,q)$, then $b_{maxprim}(G) = n+1-\delta_{2,q}$.
\item[\upshape{(ii)}] If $G = \PSU(n,q)$ with $(q+1) \nmid n$, and with $n \in \{3,5\}$ or $n \ge 16$, then $b_{maxprim}(G) = n$.
\item[\upshape{(iii)}] If $G = \PSp(n,q)$ with $n \ge 14$, then $b_{maxprim}(G) = n$.
\item[\upshape{(iv)}] If $G = \POm^\varepsilon(n,q)$ with $n \ge 16$, with $q$ even if $n = 16$, and with $G \notin \mathcal{O}$, then $b_{maxprim}(G) = n-1$.
\end{enumerate}
\end{corollary}

\begin{remark}
\label{rem:lowb}
{\rm 
Observe that if $G$ is one of the groups in Corollary~\ref{c:exactprim} for which $b_{maxprim}(G)$ is known exactly, then the upper bound for the closure number $k(G)$ from Theorem~\ref{t:rest} is equal to $b_{maxprim}(G)+1$. 
Hence this theorem gives the best upper bound for $k(G)$ that can be obtained solely using Corollary~\ref{c:redn}. On the other hand, for many low-dimensional unitary, symplectic and orthogonal groups $G$, improved upper bounds 
on the base sizes of primitive actions of $G$ on subspaces of dimension greater than two would likely result in improved upper bounds for $b_{maxprim}(G)$ and hence for $k(G)$.
}
\end{remark}


\section{$k$-closures of sporadic groups}

We start with the following lemma on certain simple groups.
 
 \begin{lemma}\label{lem:complete}
Let $G$ be a finite nonabelian simple group acting primitively on a set $\Omega$ such that $\Out(G)=1$ and $G$ is maximal in $\Alt(\Omega)$. If $G$ is $k$-transitive on $\Omega$ but not $(k+1)$-transitive for some integer $k$ such that $1\leqslant k < |\Omega|-2$, then $G^{(k+1),\Omega}=G$  but $G^{(k),\Omega}\neq G$.
 \end{lemma}
 \begin{proof}
 Note that since $G$ is $k$-transitive we have that $G^{(k),\Omega}=\Sym(\Omega)\neq G$.
  Since $G$ is maximal in $\Alt(\Omega)$ it follows that $G^{(k+1),\Omega}\cap \Alt(\Omega)=G$ or $\Alt(\Omega)$. However, $\Alt(\Omega)$ is $(k+1)$-transitive on $\Omega$ while $G$ is not, so $\Alt(\Omega)\not\leqslant G^{(k+1),\Omega}$. Hence   $G^{(k+1),\Omega}\cap \Alt(\Omega)=G$. Thus $|G^{(k+1),\Omega}:G|\leqslant 2$. Since $G$ acts primitively on $\Omega$ (and $|G|$ is not prime) we have that $C_{\Sym(\Omega)}(G)=1$. Then as  $\Out(G)=1$ it follows that $G^{(k+1),\Omega}=G$. 
  \end{proof}

 \begin{corollary}\label{cor:M23on23}
 \begin{enumerate}
     \item[\upshape{(a)}]  Let $G=M_{23}\leqslant \Sym(\Omega)$ with $|\Omega|=23$. Then $G^{(5),\Omega}=G$.
     \item[\upshape{(b)}] Let $G=Co_3\leqslant \Sym(\Omega)$.  with $|\Omega|=276$. Then $G^{(3),\Omega}=G$.
          \item[\upshape{(c)}] Let $G=Co_2\leqslant \Sym(\Omega)$ with $|\Omega|=2300$. Then $G^{(2),\Omega}=G$.
 \end{enumerate}
 \end{corollary}
 \begin{proof}
 By \cite{LPS87}, $G$ is maximal in $\Alt(\Omega)$. Additionally, $\Out(G)=1$, and so the result follows from Lemma \ref{lem:complete}.
 %
%
%
 \end{proof}

 When $G=M_{23}$, $Co_2$ or $Co_3$ we have from \cite[Table 1]{BOW} that $b_{maxprim}(G)=6$ so our next result is a slight improvement on what can be deduced from Corollary \ref{c:redn}. We suspect that it is not best possible.

 
 \begin{lemma}\label{lem:M23}
 Let $G=M_{23}$, $Co_2$ or $Co_3$. Then $k(G)\leqslant 6$.
 \end{lemma}
 \begin{proof}
 Let $n=23$ if $G=M_{23}$, $n=276$ if $G=Co_3$ and $n=2300$ if $G=Co_2$, and let $G$ act faithfully on a set $\Omega$. We will prove that $G^{(6),\Omega}=G^\Omega\cong G$, and from this it will follow that $k(G)\leqslant 6$. If $G$ has an orbit with a maximal block system $\mathcal{B}$ such that $|\mathcal{B}|\ne n$, then $b(G^{\mathcal{B}})\leqslant 5$, by \cite[Theorem 1]{BOW}, and so $b(G^\Omega)\leqslant 5$, by Lemma \ref{l:redn}. Hence by Lemma \ref{l:W}, $G^{(6),\Omega}=G^\Omega\cong G$. Thus it remains to consider the case where on each orbit $\Delta$, each maximal block system has size $n$. Let $\Delta$ be an orbit of $G$ and let $\mathcal{B}$ be a maximal block system on $\Delta$.
 By Lemma \ref{lem:onblocks} and Corollary \ref{cor:M23on23} we have that $G^{(6),\Delta}$ preserves $\mathcal{B}$ and  $(G^{(6),\Delta})^{\mathcal{B}}=G^{(6),\mathcal{B}}=G^\mathcal{B}\cong G$. Moreover, by \cite[Theorem 1]{BOW} there exist $B_1,\ldots,B_6\in\mathcal{B}$ such that $G_{B_1,\ldots,B_6}=1$ (since $G$ is faithful on $\mathcal{B}$), and it follows from Lemma \ref{l:kernel} that $G^{(6),\Delta}$ acts faithfully on $\mathcal{B}$. Thus $G^{(6),\Delta}\cong G$ so $G^{(6),\Delta}=G^\Delta$.
 
 If $\Omega=\Delta$ then we are done so we may assume that $G$ has more than one orbit. Recall that, on each $G$-orbit, each maximal block system has size $n$. Let $\Delta_1$ be a $G$-orbit and let $\mathcal{B}_1$ be a maximal block system for $G$ on $\Delta_1$. Let $\beta\in B\in\mathcal{B}_1$. Then  $G_\beta\leqslant G_B$. When $G=M_{23}$ we have that $G_B=M_{22}$, when $G=Co_3$ we have $G_B=McL.2$ and when $G=Co_3$ we have $G_B=\PSU_6(2).2$. In all cases $|G_B|$ is not divisible by $23$ but $n$ is, and so $G_B$ does not have a transitive action of degree $n$. Thus $G_B$, and hence also $G_{\beta}$ is not transitive on any orbit of $G$. Since $G^{(6),\Delta}=G^\Delta$ for each $G$-orbit $\Delta$, it follows from Lemma \ref{lem:simpleintrans}  that $G^{(6),\Omega}=G^\Omega\cong G$ and the result follows.
 \end{proof}

We use Lemma~\ref{lem:M23} for $M_{23}$, as well as the fundamental Lemma~\ref{lem:obvious}, to determine the closure number of $M_{24}.$

\begin{lemma}\label{lem:M24}
$ k(M_{24})=6$.
\end{lemma}
\begin{proof}
Let $G=M_{24}$ act faithfully on a set $\Omega$. We divide the orbits $\Delta$ of $G$ into three types:
\begin{enumerate}
\item[\upshape{(a)}] $|\Delta|=24$.
\item[\upshape{(b)}] $G$ has a maximal block system on $\Delta$ of size other than 24.
\item[\upshape{(c)}] Each maximal block system of $G$ on $\Delta$ has size 24 and at least one of these maximal block systems is not the partition into singletons. 
\end{enumerate}
Recall that we allow a maximal block system to be the partition into singletons.

Let $\Delta$ be an orbit of $G$ and suppose first that $\Delta$ is of type (a).
Then $G$ acts 5-transitively but not 6-transitively on $\Delta$. By \cite{LPS87} $G$ is a maximal subgroup of $A_{24}$ and so by Lemma \ref{lem:complete}, $k(G) \ge 6$ and $G^{(6),\Delta}=G^\Delta \cong G$. 

Next suppose that $\Delta$ has type (b) and let $\mathcal{B}$ be a maximal block system  on $\Delta$ with $|\mathcal{B}|\ne 24$. Then by \cite[Theorem 1]{BOW} the base size $b(G^{\mathcal{B}})\leqslant 4$. Hence by Lemma \ref{l:redn}, also $b(G^\Delta)\leqslant 4$ and so by Lemma \ref{l:W}, $G^{(5),\Delta}=G^\Delta\cong G$. This implies that $G^{(6),\Delta}=G^\Delta\cong G$.

Finally suppose that $\Delta$ is of type (c) and let  $\mathcal{B}$ be a maximal block system of size 24 that is not the partition into singletons. Let $B \in \mathcal{B}$. By Lemma \ref{lem:onblocks} and our result on orbits of type (a), we have $M_{24}\cong G^{\mathcal{B}}\leqslant (G^{(6),\Delta})^{\mathcal{B}}\leqslant G^{(6),\mathcal{B}}\cong M_{24}$, and so $(G^{(6),\Delta})^{\mathcal{B}}\cong G$. Moreover, by Lemma \ref{lem:GBB}, $M_{23}\cong (G^\Delta)_B^B\leq (G^{(6),\Delta})_B^B\leqslant (G_B)^{(6),B}=M_{23}^{(6),B}$, which is isomorphic to $ M_{23}$ by Lemma \ref{lem:M23}. Thus $G^{(6),\Delta}\leqslant M_{23}\wr M_{24}$ acting imprimitively. Furthermore, since the minimal degree of a transitive representation of  $M_{23}$ is 23 (see for example \cite[p71]{atlas}) we have  $|B|\geqslant 23$. Let $K$ be the kernel of the action of $G^{(6),\Delta}$ on $\mathcal{B}$ and suppose that $K\neq 1$. Then by a lemma of Scott (\cite[p.\,328]{scott} or see \cite[Theorem 4.16]{PS}) the group $K$ is a direct product of $\ell$ diagonal subgroups $H_j\cong M_{23}$ for $j=1,\ldots, \ell$ (where $\ell\geq 1$), with pairwise disjoint supports, that is to say, there is a partition $\{\mathcal{O}_1,\ldots,\mathcal{O}_\ell\}$ of $\mathcal{B}$ such that, for each $j$, $(H_j)^B\cong M_{23}$ for all $B\in\mathcal{O}_j$ and $(H_j)^B=1$ for all $B\notin\mathcal{O}_j$. Suppose that $\ell\geq2$. Then there exist distinct $B_j,B_k\in \mathcal{B}$ such that for all $\alpha_1,\ldots,\alpha_5\in B_j$ and all $\beta_1,\beta_2\in B_k$, there exists $h\in K$ such that $(\alpha_1,\ldots,\alpha_5)^h=(\alpha_1,\ldots,\alpha_5)$ and $\beta_1^h=\beta_2$. 
 Since $K\leqslant G^{(6),\Delta}$, there exists $g\in G$ such that $(\alpha_1,\ldots,\alpha_5,\beta_1)^g=(\alpha_1,\ldots,\alpha_5, \beta_1)^h=(\alpha_1,\ldots,\alpha_5,\beta_2)$. As $\beta_1, \beta_2$ were arbitrary points of $B_k$, it follows that $G_{\alpha_1,\ldots,\alpha_5}$ acts transitively on $B_k$. Note that $G_{B_j}=M_{23}$, so $G_{B_j}$ acts faithfully on $B_j$. Now $G_{\alpha_1,\ldots,\alpha_5}= (G_{B_j})_{\alpha_1,\ldots,\alpha_5}= (M_{23})_{\alpha_1,\ldots,\alpha_5}$ while $|B_k|\geqslant 23$. Suppose first that $G_{B_j}$ has a maximal block system 
 of size greater than 23 on $B_j$. Then by \cite[Theorem 1]{BOW} and Lemma \ref{l:redn}, we have $b(G_{B_j}^{B_j})\leq 4$. Hence we could have chosen $\alpha_1,\ldots,\alpha_5$ such that $(M_{23})_{\alpha_1,\ldots,\alpha_5}=1$, contradicting $G_{\alpha_1,\ldots,\alpha_5}$ acting transitively on $B_k$.  Hence all maximal block systems $\Sigma$ of $G_{B_j}$ on $B_j$ have size 23.  However, in this case
we can choose $\alpha_1,\ldots, \alpha_5\in B_j$ from pairwise distinct blocks of $\Sigma$ such that $|(M_{23})_{\alpha_1,\ldots,\alpha_5}|<23$, again contradicting $G_{\alpha_1,\ldots,\alpha_5}$ acting transitively on $B_k$.

Thus $\ell=1$ so $K\cong M_{23}$, and $G^{(6),\Delta}=M_{23}\times M_{24}$. Hence $G^{(6),\Delta}$ has an intransitive normal subgroup $N$ isomorphic to $M_{24}$ such that the normal subgroup $L\cong M_{23}$ acts transitively on the set $\Sigma$ of $N$-orbits. Now $\Sigma$ is a $G$-invariant partition of $\Delta$, and a maximal block system for the $G$-action on $\Sigma$ corresponds to a maximal block system, say $\mathcal{B}'$, for $G^\Delta$. By assumption all maximal block systems for $G$ on $\Delta$ have size 24 and so $|\mathcal{B}'|=24$. Hence the $G$-action on $\Sigma$ has a block system of size $24$. Thus the subgroup $L$ of $G$ is transitive on $\Sigma$ preserving a block system of size $24$.  This implies that $L\cong M_{23}$ has a subgroup of index $24$, which is a contradiction (see for example \cite[p71]{atlas}). 
Therefore $K=1$ and  $G^{(6),\Delta}=M_{24}$. 

Thus we have shown that $G^{(6),\Delta}=G^\Delta\cong G$ for each orbit $\Delta$.
In particular, if $G$ is transitive on $\Omega$ then we have shown that  $G^{(6),\Omega}=G^{\Omega}\cong G$.  Suppose now that $G$ is intransitive on $\Omega$.  If there is some orbit $\Delta$ of type (b) then we saw above that $b(G^\Delta)\leqslant 4$. Since $G$ acts faithfully on $\Delta$ we have $b(G^{\Omega})\leqslant 4$ and so by Lemma \ref{l:W}, $G^{(5),\Omega}=G^\Omega\cong G$.  If all orbits are of type (a) then all actions of $G$ on orbits are pairwise equivalent and so by Corollary \ref{cor:alltransequiv} we have $G^{(6),\Omega}=G^\Omega\cong G$. 

It remains to consider the case where there is at least one orbit $\Delta$ of type (c) and all the remaining orbits are of type (a) or (c). Recall that this implies that $G^\Delta$ preserves a block system $\mathcal{B}$ of size $24$ on $\Delta$, and moreover in each $G$-orbit $\Delta'$, each maximal block system for $G$ on $\Delta'$ has size $24$. Let $\beta\in B\in\mathcal{B}$. Then $G_\beta\leqslant G_B= M_{23}$. Since $M_{23}$ does not have any transitive actions of degree 24, it follows that $G_B$, and hence also $G_{\beta}$, is not transitive on any orbit of $G$. Thus by Lemma \ref{lem:simpleintrans},  and the fact that $G^{(6),\Delta'}=M_{24}$ for all $G$-orbits $\Delta'$, we deduce that $G^{(6),\Omega} \cong G$. Thus we have shown that $k(G)\leqslant 6$. Moreover,  equality holds since if $|\Omega|=24$ then $G$ is 5-transitive on $\Omega$ and hence $G^{(5),\Omega}=\Sym(\Omega)$.
\end{proof}

We can now complete the proof of Theorem \ref{t:rest}.

\medskip
\noindent\emph{Proof of Theorem \ref{t:rest}(a).}
 Let $G$ be a sporadic simple group. If $G\cong M_{24}$ then by Lemma \ref{lem:M24} we have $k(G)=6$, so suppose further that $G\not\cong M_{24}$. By Corollary~\ref{c:redn}, the closure number $k(G)$ satisfies $k(G)\leq b_{maxprim}(G)+1$, where $b_{maxprim}(G)$ is as in \eqref{e:bmaxprim}.  Let $\Omega$ be a finite set on which $G$ acts faithfully and primitively such that $b(G^\Omega) = b_{maxprim}(G)$, so $k(G)\leq b(G^\Omega)+1$. By \cite[Theorem 1]{BOW} we in fact have that either $b(G^{\Omega})\leqslant 5$ and so $k(G)\leqslant 6$, or $G=M_{23}$, $Co_2$ or $Co_3$. Thus Theorem \ref{t:rest}(a) follows from Lemma \ref{lem:M23}. 
\hfill\qedsymbol

\end{document}